\documentclass[graybox, envcountsect,envcountsame]{svmult}

\usepackage{mathptmx}       
\usepackage{helvet}         
\usepackage{courier}        
\usepackage{type1cm}        
\usepackage{makeidx}         
\usepackage{graphicx}        
\usepackage{multicol}        
\usepackage[bottom]{footmisc}
\usepackage{changepage}
\usepackage{lipsum}

\usepackage[all]{xy}

\usepackage{tikz-cd}

\usepackage{amscd}

\usepackage{todonotes}

\usepackage{verbatim}

\usepackage{amsmath,amssymb}

\makeindex             

\def\N{\mathbb{N}}
\def\ff{\frak}
\def\Spec{\mbox{\rm Spec}\:}

\def\p{\mathfrak{p}}

\def\m{{\mathfrak m}} 

\def\dim{{\rm dim~ }}

\def\ord{{\rm ord }}

\def\edeg{{\rm e}}

\spnewtheorem{fact}[theorem]{Fact}{\it}{}

\spnewtheorem{discussion}[theorem]{Discussion}{\bf}{}

\spnewtheorem{construction}[theorem]{Construction}{\bf}{}

\begin{document}


\title*{Directed unions of local quadratic transforms  of  regular local rings and  pullbacks}

\author{Name of First Author and Name of Second Author}
\institute{Name of First Author \at Name, Address of Institute, \email{name@email.address}
\and Name of Second Author \at Name, Address of Institute \email{name@email.address}}

\author{Lorenzo Guerrieri, William Heinzer, Bruce Olberding and Matt  Toeniskoetter}
\institute{Lorenzo Guerrieri \at Università di Catania, Viale Andrea Doria 6, 95125 Catania, Italy \email{guerrieri@dmi.unict.it}
\and William Heinzer \at 
Department of Mathematics, Purdue University, West Lafayette, Indiana 47907
\email{heinzer@purdue.edu}
\and
Bruce Olberding \at New Mexico State University, Department of Mathematical Sciences, Las Cruces, NM 88003-8001 
\email{olberdin@nmsu.edu}
\and Matt Toeniskoetter \at 
Department of Mathematics, Purdue University, West Lafayette, Indiana 47907
\email{mtoenisk@purdue.edu}
}


%

\maketitle

\abstract*{Each chapter should be preceded by an abstract (10--15 lines long)}

\abstract{Let  $\{ R_n, \m_n \}_{n \ge 0}$ be an  infinite sequence of regular local rings with $R_{n+1}$ birationally dominating $R_n$  and $\m_nR_{n+1}$ a principal ideal of $R_{n+1}$  for 
each $n$.    We  examine properties of the integrally closed local domain $S = \bigcup_{n \ge 0}R_n$. }


{\noindent}{\bf Keywords} $\:$  $\bullet$ regular local ring  $\bullet$  local  quadratic transform 
$\bullet$   valuation ring $\bullet$     pullback construction
$\:$ 

{\noindent}{\bf Mathematics Subject Classification (2010):} 
{13H05,  13A15, 13A18}

\section{Introduction}

Let $R$ be a regular local ring with maximal ideal ${\ff m} = (x_1,\ldots,x_n)R$, where $n = \dim R$ is the Krull dimension of $R$.  Choose $i \in \{1,\ldots,n\}$, and consider the overring $R[\frac{x_1}{x_i},\ldots,\frac{x_n}{x_i}]$ of $R$. 
Choose any prime ideal $P$ of $R[\frac{x_1}{x_i},\ldots,\frac{x_n}{x_i}]$ that contains ${\ff m}$. Then the ring
 $R_1 := R[\frac{x_1}{x_i},\ldots,\frac{x_n}{x_i}]_P$ is a {\it local quadratic transform} of $R$, and 
 $R_1$ is again a regular local ring and $\dim R_1 \le n$. 
 Iterating the process we obtain a sequence $R = R_0 \subseteq R_1 \subseteq R_2 \subseteq \cdots $ of regular local overrings of $R$ such that for each $i$, $R_{i+1}$ is a local quadratic transform of $R_i$. 
 The sequence of positive integers 
  $\{\dim R_i\}_{i \in \N}$   stabilizes, and  $\dim R_i = \dim R_{i+1}$ for  
  all sufficiently large $i$.  If $\dim R_i = 1$,  then necessarily $R_i = R_{i+1}$,
  while if $\dim R_i \ge 2$,  then $R_i \subsetneq R_{i+1}$.


 The process of iterating local quadratic transforms of the same Krull dimension is the algebraic expression of following a closed point through a sequence of blow-ups of a nonsingular point of an algebraic  variety, 
 with each blow up occurring at a closed point in the fiber of the previous blow-up. 
  This geometric process plays a central role in embedded resolution of singularities for curves on surfaces (see, for example, \cite{Abh4} and \cite[Sections 3.4 and 3.5]{Cut}), 
 as well as factorization of birational morphisms between nonsingular surfaces (\cite[Theorem 3]{Abh} and \cite[Lemma, p.~538]{ZarR}).  These applications  depend on properties of  iterated sequences of local quadratic transforms of a two-dimensional regular local ring.  For  a two-dimensional regular local  ring $R$, 
  Abhyankar \cite[Lemma 12]{Abh} shows   that  the limit of this process of iterating local quadratic transforms 
  $R = R_0 \subseteq R_1 \subseteq R_2 \subseteq \cdots $ results in a valuation ring that birationally dominates $R$; i.e., ${\cal V} = \bigcup_{i =0}^{\infty}R_i$ is a valuation ring  with the same quotient field  as $R$ and the maximal ideal of $\mathcal V$
  contains     the maximal ideal of $R$.  
  
  Moving beyond dimension two, examples due to David Shannon \cite[Examples 4.7 and 4.17]{Sha} show that the union $S= \bigcup_{i}R_i$ 
  of an iterated  sequence of local quadratic transforms of a regular local ring of Krull dimension $>2$ 
  need not be a valuation ring.      The recent articles \cite{HLOST,HOT} address the structure of such rings 
  and how this  structure 
  encodes asymptotic properties of the sequence $\{R_i\}_{i=0}^\infty$. 
  We call  $S$  a {\it quadratic Shannon extension} of $R$. 
  In general, a quadratic Shannon extension  need not be a valuation ring nor a Noetherian ring, although it is always an intersection of two such rings (see Theorem~\ref{hull}).   
   
    The class of quadratic Shannon extensions separates naturally into two cases, the archimedean and non-archimedean cases. A quadratic Shannon extension $S$ is non-archimedean if there is an element $x$ in the maximal ideal of $S$ such that $\bigcap_{i>0}x^iS \ne 0$.  The class of non-archimedean
    quadratic  Shannon extensions is analyzed in  detail in \cite{HLOST} and \cite{HOT}. 
     We carry this analysis further in the present article by using  techniques from 
    multiplicative ideal theory to classify  a non-archimedean quadratic Shannon extension as the 
     pullback of a valuation ring of rational rank one along a homomorphism from a regular local 
     ring onto its residue field. 
 We present  several   variations of  this classification in Lemma~\ref{3.6} and Theorems~\ref{pullback thm} and~\ref{pullback cor}.
   
     The pullback description leads in Theorem~\ref{3.8} to existence results for both archimedean and non-archimedean quadratic Shannon extensions  contained in a localization of the base ring at a nonmaximal prime ideal.  As another application,  in Theorem~\ref{function field} we use pullbacks to characterize the quadratic Shannon extensions $S$ of regular local rings $R$ such that  $R$  is essentially finitely generated over a field of characteristic $0$ and $S$ has a principal maximal ideal. 
    
      That non-archimedean quadratic Shannon extensions occur as pullbacks is also useful because of the extensive literature on 
transfer properties between the rings  in a pullback square.   
In   Section \ref{section 6}   we use the pullback 
classification along with structural results for archimedean quadratic Shannon extensions 
from \cite{HLOST} to show in Theorem~\ref{quadratic GCD}  that a quadratic Shannon extension 
is coherent if and only if it is 
a valuation domain.


Our methods sometimes involve local quadratic transforms of Noetherian local domains that need not be 
regular local rings. 
To formalize these notions as well as those mentioned above,   let $(R,\m)$ be a Noetherian local domain 
and let $(V, \m_V)$ be a valuation domain birationally dominating $R$.
Then $\m V = xV$ for some $x \in \m$.
The ring $R_1 = R[\m/x]_{\m_V \cap R[\m/x]}$ is called a 
{\it local quadratic transform} (LQT) of $R$ along $V$.
The ring $R_1$ is a Noetherian local domain that dominates $R$ with maximal ideal
 $\m_1 = \m_V \cap R_1$.
Since $V$ birationally dominates $R_1$, we may iterate this process 
to obtain an infinite sequence $\{ R_n \}_{n \ge 0}$ of LQTs of $R_0 = R$ along $V$.
If $R_n = V$ for some $n$, then $V$ is a DVR and the sequence stabilizes with $R_m = V$ for all $m \ge n$.
Otherwise, $\{R_n\}$ is an infinite strictly ascending sequence of Noetherian local domains.

If $R$ is a regular local ring (RLR), it is well known that $R_1$ is an RLR; cf. \cite[Corollary 38.2]{Nag}.
Moreover, $R = R_1$ if and only if $\dim R \le 1$. 
Assume that $R$ is an RLR with $\dim R \ge 2$ and $V$ is minimal as a valuation overring of $R$.
Then $\dim R_1 = \dim R$, and the process may be continued 
by defining $R_2$ to be the LQT of $R_1$ along $V$.
Continuing the procedure yields an infinite strictly ascending sequence $\{R_n\}_{n \in \N}$
of RLRs all dominated by $V$.

  In general,  our notation is as in Matsumura   \cite{Mat}.  Thus a local ring need not be Noetherian.
 An element $x$ in the maximal ideal $\m$ of a regular local ring $R$ is said
to be a {\it regular parameter} if $x \not\in \m^2$.   It then follows that  the residue class ring $R/xR$ is again a regular local 
ring.  
We refer to an extension ring $B$ of an integral domain $A$ as an {\it overring of} $A$ if $B$ is a subring of the quotient field of $A$.  If, in addition,  $A$ and $B$ are local and the inclusion map $A \hookrightarrow B$ is a 
local homomorphism,  we say that $B$ {\it birationally dominates}  $A$. 
 We use UFD as an abbreviation for unique factorization domain,  and DVR as an abbreviation for rank 1 discrete valuation ring.  If $P$ is a prime ideal of a ring $A$, we denote by $\kappa(P)$ the residue field $A_P/PA_P$ of $A_P$.

\section{Quadratic Shannon extensions} 

Let $(R,\m)$ be a regular local ring  with $\dim R \ge 2$ and let $F$ denote the quotient field of $R$. 
David Shannon's work in \cite{Sha}  on sequences of quadratic and monoidal transforms of 
regular local rings  motivates  our terminology  quadratic Shannon extension  in  Definition~\ref{1.1}.

\begin{definition} \label{1.1}
  Let $(R,\m)$ be a regular local ring with $\dim R \ge 2$. With $R = R_0$,
  let $\{R_n, \m_n\}$ be an infinite sequence of RLRs, 
  where $\dim R_n \ge 2$ for each $n$.
       If $R_{n+1}$ is an LQT of $R_n$  for each $n$,  
      then the ring $S = \bigcup_{n \ge 0}R_n$ is called a 
      {\it quadratic Shannon extension}\footnote{
        In \cite{HLOST} and \cite{HOT}, the authors call $S$ a Shannon extension of $R$. We
        have made a distinction here with monoidal transforms. 
        Since $\dim R_n \ge 2$, we have  $R_n \subsetneq  R_{n+1}$ for each positive integer $n$ 
        and $\bigcup _n R_n$ is an infinite ascending union.
      } of $R$.
    \end{definition}

If $\dim R = 2$, then the quadratic Shannon extensions of $R$ are precisely the valuation rings that birationally dominate $R$ and are minimal as a valuation overring of $R$ \cite[Lemma 12]{Abh}.  
If $\dim R > 2$, then,  examples due to Shannon \cite[Examples~4.7 and~4.17]{Sha} show
that   there are quadratic Shannon extensions that are not valuation rings. 
Similarly, if $\dim R >2$, then there are valuations rings $V$ that birationally dominate $R$ 
with $V$ minimal as a valuation overring of $R$,   but $V$ is not a Shannon extension of $R$.
 Indeed, if $V$ has rank $>2$, then $V$ is not a quadratic Shannon extension 
 of $R$; see \cite[Proposition 7]{Gra}.   

These observations raise the question of   the ideal-theoretic structure of a quadratic Shannon extension of a regular local ring $R$ with $\dim R > 2$, a question that was taken up in \cite{HLOST} and \cite{HOT}. 
  In this section we recall some of the results from \cite{HLOST} and \cite{HOT} with special emphasis on non-archimedean quadratic Shannon extensions, a class of  Shannon extensions that we classify in Sections~\ref{section 3} and \ref{section 3.5}.

To each quadratic Shannon extension  
there is an associated 
 collection of rank 1 discrete valuation rings. 
Let $S  = \bigcup_{i \ge 0}R_i$ be a  quadratic Shannon extension of $R = R_0$. 
For each $i$, let $V_i$ be the  DVR  
defined by  the {\it order function}  $\ord_{R_i}$,  where for $x \in R_i$, 
$\ord_{R_i}(x) = \sup \{ n \mid x \in \m_i^n \}$ and $\ord_{R_i}$ is extended to the quotient field of $R_i$ by defining $\ord_{R_i}(x/y) = \ord_{R_i}(x) -\ord_{R_i}(y)$ for all $x,y \in R_i$ with $y \ne 0$. 
The family $\{V_i\}_{i=0}^\infty$ determines 
a unique set  
\begin{equation*}\label{equation V}
	V ~   =   ~    \bigcup_{n \ge 0}~ \bigcap_{i \ge n} V_i = \{ a \in F ~ | ~  \ord_{R_i }(a) \ge 0 \text{ for } i \gg 0 \}.\end{equation*}

The set $V$  consists of the elements in $F$ that are in all but finitely many of the  $V_i$.
In \cite[Corollary 5.3]{HLOST},  the authors prove that $V$ is a valuation domain that birationally
dominates $S$,  and call $V$ the {\it boundary valuation ring} of the Shannon extension $S$.

 Theorem~\ref{hull} records properties of a quadratic Shannon extension.

\begin{theorem} $\phantom{}$ \hspace{-.09in}  {\rm \cite[Theorems 4.1, 5.4 and~8.1]{HLOST}}    \label{hull} 
    Let $(S,{\ff m}_S)$ be a  quadratic  Shannon extension of a regular local ring $R$.
Let $T$ be the intersection of all the 
DVR  overrings of $R$   that properly contain $S$, and let $V$ be the boundary valuation ring of $S$.  
Then:
\label{flat} 
\begin{description}[$(2)$]

\item[{\em (1)}] 
$\dim S = 1$     if and only if   
$S$ is a rank 1 valuation ring.

\item[{\em (2)}]
$S = V \cap T$.  

\item[{\em (3)}]
 
There exists $x \in \m_S$ such that $xS$ is $\m_S$-primary, and $T = S [1/x]$ for any such $x$.
It follows that the units of $T$ are precisely the ratios of $\m_S$-primary elements of $S$ and 
$\dim T = \dim S - 1$.

\item[{\em (4)}]

$T$ is a localization of $R_i$ for $i \gg 0$. In particular, $T$ is a Noetherian regular UFD.

\item[{\em (5)}]

$T$ is the unique minimal proper Noetherian overring of $S$. 

 \end{description} 
\end{theorem} 

In light of item 5 of Theorem~\ref{hull}, the ring $T$  is called the {\it Noetherian hull} of $S$.  

\medskip

The boundary valuation ring is given by a valuation from the nonzero elements of the  quotient field of $R$ to a totally ordered abelian group of rank at most $2$ \cite[Theorem 6.4 and Corollary 8.6]{HOT}.   In \cite{HOT} the following two mappings on the quotient field of $R$ are introduced as invariants of a quadratic Shannon extension. The first, $\edeg$, takes values in ${\mathbb{Z}} \cup \{\infty\}$, while the second, $w$, takes values in ${\mathbb{R}} \cup \{-\infty,+\infty\}$.  Both $e$ and $w$ are used in \cite{HOT} to decompose the  boundary valuation $v$ of the quadratic Shannon extension into a function that takes its values in ${\mathbb{R}} \oplus {\mathbb{R}}$ with the lexicographic ordering. The function $\edeg$ is defined in terms of the transform $(aR_n)^{R_{n+i}}$ of a principal ideal $aR_n$ in $R_{n+i}$ for $i > n$; see \cite{Lip} for the general definition of the (weak) transform of an ideal and \cite{HLOST} for more on the properties of the transform in our setting.

\begin{definition} \label{e def} \label{w-function}
Let $S = \bigcup_{i \geq 0} R_i$ be a quadratic Shannon extension of a regular local ring $R$.  
\begin{description}[(2)]
\item[(1)] Let $a \in S$ be nonzero.  
Then $a \in R_n$ for some $n \ge 0$.  
Define $$\displaystyle \edeg (a) = \lim_{i \rightarrow \infty} \ord_{n +i} ((a R_n)^{R_{n+i}}).$$
For   $a, b$  nonzero elements in $ S$, let $n \in \mathbb N$ be 
such that $a, b \in R_n$ and 
 define $\edeg (\frac{a}{b}) = \edeg (a) - \edeg (b)$. That $\edeg$ is well defined is given by \cite[Lemma 5.2]{HOT}.  
 
\item[(2)] 
Fix $x \in S$ such that $xS$ is primary for the maximal ideal of $S$, and define  
$$w:~F \rightarrow {\mathbb{R}} ~  \cup   ~  \{-\infty,~+ \infty \}$$ by   
defining  $w (0) = +\infty$, and for each $q \in F^{\times}$,  
	$$
	w (q) ~ = ~  \lim_{n \rightarrow \infty} \frac{\ord_n (q)}{\ord_n (x)}.
	$$
	\end{description}
\end{definition}

The structure of Shannon extensions naturally separate into those that are archimedean and those that are 
non-archimedean as in the following definition.

\begin{definition} {\rm An integral domain $A$ is {\it archimedean} if $\bigcap_{n>0} a^n A = 0$ for each nonunit $a \in A$.} 
\end{definition}

An integral domain $A$ with $\dim A \le 1$ is archimedean.   

Theorem~\ref{overview}, which characterizes quadratic Shannon extensions in several ways,  shows that there is a prime ideal $Q$ of a non-archimedean quadratic Shannon extension $S$ such that $S/Q$ is a rational rank one valuation ring and $Q$  is  a prime ideal of the Noetherian hull $T$ of $S$. In the next section this fact serves as the basis for the classification of non-archimedean quadratic Shannon extensions via pullbacks.

\begin{theorem}  \label{overview} 
Let $S = \bigcup_{n \geq 0}R_n$ be a quadratic 
	Shannon extension of a regular local ring $R$ with quotient field $F$, and let $x$ be an element of $S$ that is primary for the 
	maximal ideal  $\m_S$ of $S$ (see Theorem~\ref{flat}). 
	Assume  that $\dim S \ge 2$. Let $Q = \bigcap_{n \ge 1}x^nS$, and 
	let $T = S[1/x]$ be the Noetherian hull of $S$. 
	Then the following are equivalent: 
	\begin{description}[(2)]
		\item[{\rm (1)}]  $S$ is non-archimedean. 
		
		\item[{\rm (2)}]   $T = (Q:_FQ)$.  
		
		\item[{\rm (3)}]  $Q$ is a nonzero prime ideal of $S$.  
		
		\item[{\rm (4)}] 
		Every nonmaximal prime ideal  of $S$ is contained in $Q$.
		
		\item[{\rm (5)}] 
		$T$ is a regular local ring. 
		
		\item[{\em (6)}] $\sum_{n=0}^{\infty} w ({\ff m}_n) = \infty$, 
		where $w$ is as in Definition~\ref{w-function} and, 
		for each $n \geq 0$, 
		${\ff m}_n$ is the maximal ideal of $R_n$.

	\end{description}
	Moreover if (1)--(6) hold for $S$ and $Q$, then $T = S_Q$, $Q = QS_Q$   is a common ideal of $S$ and $T$,  and $S/Q$ is   a  
	rational rank 1 valuation domain  on the residue field $T / Q$ of $T$.  In particular, $Q$ is the unique maximal ideal of $T$.  
\end{theorem}

\begin{proof} The equivalence of items 1 through 5 can be found in \cite[Theorem 8.3]{HOT}. That statement 1 is equivalent to 6 follows 
	from  \cite[Theorem 6.1]{HOT}. 
	To prove the moreover  statement, 
	define $Q_\infty = \{ a \in S \mid w (a) = + \infty \},$ where $w$ is as in Definition~\ref{w-function}. 
	By \cite[Theorem~8.1]{HOT}, $Q_\infty$ is a prime ideal of $S$ and $T$, and
	by  \cite[Remark~8.2]{HOT}, $Q_\infty$ 
	is the unique prime ideal of $S$ of dimension 1.  Since also item 4 implies every nonmaximal prime ideal of $S$ is contained in $Q$, it follows that $Q = Q_{\infty}$.  By item  5, $T = S[1/x]$ is a local ring. Since $xS$ is ${\ff m}_S$-primary, we have that $T = S_Q$.   Since $Q$ is an ideal of $T$, we conclude that $QS_Q = Q$ and $Q$ is  the unique maximal ideal of $T$.  By \cite[Corollary 8.4]{HOT}, $S/Q$ is a valuation domain, and by  \cite[Theorem~8.5]{HOT}, $S/Q$ has rational rank 1. 
 \qed
\end{proof}

We can further separate the case where $S$ is archimedean to whether or not $S$ is completely integrally closed.
We recall the definition and result.

\begin{definition} \label{2.5}
 {\rm Let $A$ be an integral domain. An element $x$ in the field of fractions of $A$ is called {\it almost integral} over $A$ if $A [x]$ is contained
  in a principal fractional ideal of $A$.
    The ring $A$ is called {\it completely integrally closed} if it contains all of the almost integral elements over it.} 
\end{definition}  

\begin{theorem} $\phantom{}$ \hspace{-.09in}  {\rm \cite[Theorems 6.1, 6.2]{HLOST}}    \label{complete intcl} 
  Let $S$ be an archimedean quadratic Shannon extension.
  Then the function $w$ as in Definition~\ref{w-function} is a rank $1$ nondiscrete valuation.
  Its valuation ring $W$ is the rank $1$ valuation overring of $V$ and $W$ also dominates $S$.
  The following are equivalent:
  \begin{description}
    \item[{\rm (1)}] $S$ is completely integrally closed.
    \item[{\rm (2)}] The boundary valuation $V$ has rank $1$; that is, $V = W$.
  \end{description}
\end{theorem}

In Theorem~\ref{nonarch valuation} we recall from \cite{HOT} the decomposition of the boundary valuation of a non-archimedean quadratic Shannon extension in terms of the functions $w$ and $\edeg$ in Definition~\ref{w-function}. For a  decomposition of the boundary valuation in terms of $w$ and $\edeg$ in the archimedean case, see \cite[Theorem 6.4]{HOT}.

\begin{theorem} $\phantom{}$ \hspace{-.09in}  {\rm \cite[Theorem 8.5 and Corollary 8.6]{HOT}}
 \label{nonarch valuation}
Assume that   $S$ is a non-arch\-i\-med\-e\-an quadratic 
	Shannon extension of a regular local ring $R$ with quotient field $F$. 
	 Let $Q$ be as in Theorem~\ref{overview}, and 
	let $e$ and $w$ be as in Definition~\ref{w-function}. Then:
\begin{description}[(2)]
		\item[{\em (1)}]
		
		$\edeg$ is a rank 1 valuation on $F$ whose valuation ring $E$ contains $V$. If in addition $R / (Q \cap R)$ is a regular local ring, then $E$ is the order valuation ring of $T$.
		
		\item[{\em (2)}]
		
		$w$ induces a rational rank $1$ valuation $w'$ on the residue field $E/\mathfrak  m_E$  of $E$. The valuation ring  $ W'$  defined by $w'$
		extends the valuation ring $S / Q$,
		and the value group of $W'$ is the same as the value group of $S / Q$.
		
		\item[{\em (3)}]
		$V$ is the valuation ring defined by the composite valuation of $\edeg$ and $w'$. 
		
		\item[{\em (4)}] Let $z \in E$ such that ${\ff m}_E = zE$. Then $V$ is defined by the valuation $v$ given by 
		$$v:F \setminus \{0\} \rightarrow {\mathbb{Z}} \oplus {\mathbb{Q}}:a \mapsto \left(\frac{\edeg(a)}{\edeg(z)}, \frac{w(a)\edeg(z)}{w(z)\edeg(a)}\right),$$
where the direct sum is ordered lexicographically. 		
	\end{description}
	\end{theorem}

\section{The relation of Shannon extensions to pullbacks}  \label{section 3}

Let $\alpha:A \rightarrow C$ be an extension of rings, and let $B$ be a subring of $C$. The subring  $D = \alpha^{-1}(B)$ of $A$ is the {\it pullback} of $B$ along $\alpha:A \rightarrow C$.  
\begin{center}
\begin{tikzcd}
   D\arrow[rightarrow]{r}\arrow[hookrightarrow]{d} 
  & B\arrow[hookrightarrow]{d}
  \\ 
   A\arrow{r}{\alpha} 
 &  C
\end{tikzcd}
\end{center}
Alternatively, $D$ is the fiber product  $A \times_{C} B$ of  $\alpha$ and the inclusion map $\iota:~B \rightarrow~C$;  see, for example,  \cite[page~43]{Kunz}.


The pullback construction has been extensively studied in multiplicative ideal theory, where it  serves as a   source of examples and generalizes  the classical ``$D+M$'' construction. (For more on the latter construction, see \cite{Gil}.)  We will be especially interested in the case in which $A,B,C,D$ are domains, $\alpha$ is a surjection and $B$ has quotient field $C$.  In this case, following \cite{Gabelli-Houston}, we say the diagram above is of type $\square^*$.  For a diagram of type $\square^*$, the kernel of $\alpha$ is a maximal ideal of $A$ that is contained in $D$. The quotient field $C$ of $B$  can be identified with the residue field of this maximal ideal. 
 If  $A$ is local with 
 $\dim A \ge 1$ and $\dim B \ge 1$, then   $A = D_M$  is  a localization of $D$ and  $D$ is non-archimedean. These observations  have a number of 
 consequences for transfer properties between the ring $D$ and the rings $A$ and $B$; see for example \cite{Fon, FHP, Gabelli-Houston}. 

While pullback diagrams of type $\square^*$ are often used to construct examples in non-Noetherian commutative ring theory, there are also instances where the pullback construction is used as a classification tool.  A simple example is given by the observation that a local domain $D$ has a principal maximal ideal if and only if $D$ occurs in a pullback diagram of type $\square^*$, where $B$ is a DVR \cite[Exercise~1.5, p.~7]{Kap}.   A second example is given by the fact that for nonnegative integers $k< n$,  a ring $D$ is a valuation domain of rank $n$ if and only if $D$ occurs in pullback diagram of type $\square^*$, where $A$ is a valuation ring of rank $n-k$ and $B$ is a valuation ring of  rank $k$; see \cite[Theorem 2.4]{Fon}.   
Theorem~\ref{overview} 
provides an instance of this decomposition in the present context. In the theorem, $V$ is the pullback of $E$ and $W'$: 
\begin{center}
\begin{tikzcd}
   V = \alpha^{-1}(W')\arrow[twoheadrightarrow]{r}\arrow[hookrightarrow]{d} 
  & W'\arrow[hookrightarrow]{d}
  \\ 
   E\arrow[twoheadrightarrow]{r}{\alpha} 
 &  E/{\ff m}_E
\end{tikzcd}
\end{center}
	
	A third example classification via pullbacks of the form $\square^*$ is given by the classification of local rings of global dimension $2$ by Greenberg \cite[Corollary 3.7]{Gre} and Vasoncelos \cite{Vas}:  A local ring $D$ has global dimension $2$ if and only if $D$ satisfies one of the following:
	\begin{description}[(a)]
	\item[(a)] $D$ is a regular local ring of Krull dimension $2$, 
	\item[(b)] $D$ is a valuation ring of global dimension $2$, or 
	\item[(c)] $D$ has countably many principal prime ideals and $D$ occurs in a pullback diagram of type $\square^*$, where $A$ is  a valuation ring of global dimension 1 or 2 and $B$ is a regular local  ring of global dimension  $2$.  
	\end{description}

		Motivated by these examples, we use the pullback construction in this and the next section to classify among the overrings of a regular local ring $R$ those that are non-archimedean quadratic Shannon extensions of $R$.   We prove in Theorem~\ref{pullback thm} that these are precisely the overrings of $R$ that occur in pullback diagrams of type~$\square^*$, where $A$ is a localization of an iterated quadratic transform $R_i$ of $R$ at a prime ideal $P$ and 
	 and $B$ is a rank $1$ valuation overring of $R_i/P$ having a divergent multiplicity sequence. Thus a non-archimedean quadratic Shannon extension is 
	 determined by  a rank 1 valuation ring and a regular local ring.

	As a step towards this classification, in Theorem~\ref{pullbacks} we restate  part of Theorem~\ref{overview} as an assertion about how a non-archimedean quadratic Shannon extension can be decomposed using pullbacks. 
 Much of the rest of 
 this section and the next is devoted to a converse of this assertion, which is given in Theorems~\ref{pullback thm} and~\ref{pullback cor}.

\begin{theorem}  \label{pullbacks}   Let $ S$ be a non-archimedean quadratic Shannon extension. Then there is a prime ideal $P$ of $S$ and a rational rank $1$ valuation ring ${\cal V}$ of $\kappa(P)$ such that 
 $S_P$ is the Noetherian hull of $S$ and
 $S$ is the pullback of ${\cal V}$ along the residue map $\alpha:S_P \rightarrow \kappa(P)$, as in the following diagram:
\begin{center}
\begin{tikzcd}
   S = \alpha^{-1}({\cal V})\arrow[twoheadrightarrow]{r}\arrow[hookrightarrow]{d} 
  & {\cal V} \arrow[hookrightarrow]{d}
  \\ 
   S_P\arrow[twoheadrightarrow]{r}{\alpha} 
 &  \kappa(P)
\end{tikzcd}
\end{center}

\end{theorem}

\begin{proof}   Theorem~\ref{overview}  implies that there  is a prime ideal $P$ of $S$ such that $ S_P$ is the Noetherian hull of $S$, $P = PS_P$ and  $S/P$ is  a  rational rank 1 valuation ring. 
Theorem~\ref{pullbacks}
  follows from these observations.  \qed
\end{proof}

\begin{definition} Let $R$ be a Noetherian local domain, let  ${\mathcal{V}}$ be a rank $1$ valuation ring dominating $R$ with corresponding 
valuation $\nu$, and let $\{ (R_i,\m_i ) \}_{i=0}^\infty$ be the infinite sequence\footnote{If $R_n = R_{n+1}$ for some integer $n$, then $R_n = \mathcal V$ is a 
DVR and $R_n = R_m$ for all $m \ge n$.}
 of LQTs along $\mathcal{V}$.  Then the sequence $\{\nu(\m_i)\}_{i=0}^\infty$ is the {\it multiplicity sequence of $(R,{\mathcal{V}})$}; see \cite[Section 5]{GMR}.
	We say the multiplicity sequence is \emph{divergent} if $\sum_{i \ge 0} \nu (\m_i) = \infty$.
\end{definition}


\begin{remark} \label{GMR remark} Let $R$ be  a regular local ring and let  ${\mathcal{V}}$ be a rank $1$ valuation ring birationally dominating $R$.
	If the multiplicity sequence of $(R, \mathcal{V})$ is divergent, then ${\mathcal{V}}$ is a quadratic Shannon extension of $R$ \cite[Proposition 23]{GMR} and ${\mathcal{V}}$ has rational rank~$1$ \cite[Proposition 7.3]{HKT2}.  This is observed in \cite[Corollary~3.9]{HLOST} in the case where $\mathcal{V}$ is a DVR.	
	In Proposition~\ref{3.7}, we observe that $\mathcal{V}$ is the union of the rings in the LQT sequence of $R$ 
	along $\mathcal{V}$ for every Noetherian local domain $(R,\m)$ birationally dominated by $\mathcal {V}$.
\end{remark}

\begin{proposition}\label{3.7}
  Let $(R,\m)$ be a Noetherian local domain, let $\mathcal{V}$ be a rank $1$ valuation ring that birationally dominates $R$,
  and let $\{ R_i \}_{i = 0}^{\infty}$ be the infinite sequence of LQTs of $R$ along $\mathcal{V}$.
  If the multiplicity sequence of $(R, \mathcal{V})$ is divergent,
  then $\mathcal{V} = \bigcup_{n \ge 0} R_n$.
  In particular, if $\mathcal{V}$ is a DVR, then $\mathcal{V} = \bigcup_{n \ge 0} R_n$.
\end{proposition}

\begin{proof}
	Let $\nu$ be a valuation for $\mathcal{V}$ and let $y$ be a nonzero element in $\mathcal{V}$.
	Suppose we have an expression $y = a_n/b_n$, where $a_n, b_n \in R_n$. 	Since $R_n \subseteq \mathcal{V}$, it follows that $\nu (b_n) \ge 0$. If $\nu (b_n) = 0$, then since 
	 $\mathcal{V}$ dominates $R_n$, we have  $1/b_n \in R_n$ and $y \in R_n$.
	
	Assume otherwise, that is, $\nu (b_n) > 0$. Then $b_n \in \m_n$, and since $\nu (a_n) \ge \nu (b_n)$, also $a_n \in \m_n$.
	Let $x_n \in \m_n$ be such that $x_n R_{n+1} = \m_n R_{n+1}$.
	Then $a_n, b_n \in x_n R_{n+1}$, so the elements $a_{n+1} = a_n / x$ and $b_{n+1} = b_n / x$ are in $R_{n+1}$.
	Thus we have the expression $y = a_{n+1}/b_{n+1}$, where $\nu (b_{n+1}) = \nu (b_{n}) - \nu (\m_n)$.
	
	Consider an expression $y = a_0/b_0$, where $a_0, b_0 \in R_0$.
	Then we iterate this process to obtain a sequence of expressions $\{ a_n/b_n \}$ of $y$, with $a_n, b_n \in R_n$, where this process halts at some $n \ge 0$ if $\nu (b_n) = 0$, implying $y \in R_n$.
	Assume by way of contradiction that this sequence is infinite.
	For $N \ge 0$, it follows that $\nu (b_0) = \nu (b_N) + \sum_{n = 0}^{N-1} \nu (\m_n)$.
	Then $\nu (b_0) \ge \sum_{n = 0}^{N} \nu (\m_n)$ for any $N \ge 0$, so $\nu (b_0) \ge \sum_{n = 0}^{\infty} \nu (\m_n) = \infty$, which contradicts $\nu (b_0) < \infty$.
	This shows that the sequence $\{a_n/b_n\}$ is finite and hence $y \in \bigcup_{n}R_n$.  
	 \qed
\end{proof}

\begin{remark} \label{divmultse}
Examples  of $(R, \mathcal{V})$ with divergent multiplicity sequence
such that $\mathcal{V}$ is not a DVR are  given in \cite[Examples 7.11 and  7.12]{HKT2}.

\end{remark}

\begin{discussion}  \label{3.6disc}   
Let $(R,\m)$ be a Noetherian local domain, let $\mathcal{V}$ be a rational  rank $1$ valuation ring that birationally dominates $R$,
  and let $\{ R_i \}_{i = 0}^{\infty}$ be the infinite sequence of LQTs of $R$ along $\mathcal{V}$. 
 The divergence of the multiplicity sequence in  Proposition~\ref{3.7} is a sufficient condition for $\mathcal V = \bigcup_{n \ge 0}R_n$, but not a 
  necessary condition; see Example~\ref{3.7ex}. 
  It would be interesting to understand more about conditions in order that the multiplicity sequence of 
  $(R, \mathcal{V})$ is divergent.    Example~\ref{3.7ex} illustrates that an archimedean Shannon extension $S$ 
  of a 3-dimensional regular local ring $R$ may be
  birationally dominated by a rational rank 1 valuation ring $V$, where $S \subsetneq V$. In 
  this case by Proposition~\ref{3.7}, the multiplicity sequence of $(R, V)$ must be convergent.
\end{discussion}

\begin{example} \label{3.7ex}   
Let $x, y, z$ be indeterminates over a field $k$. We first construct a rational rank 1 valuation ring
$V'$ on the field $k(x,y)$.  We do this by describing an infinite sequence $\{(R'_n, \m'_n)\}_{n \ge 0}$
of local quadratic transforms of $R'_0 = k[x, y]_{(x,y)}$.  To indicate properties of the sequence, we
define a rational valued function $v$ on specific generators  of the $\m'_n$.  The function $v$ is to
be additive on products. 
 We set $v(x) = v(y) = 1$. 
This indicates that $y/x$ is a unit in every valuation ring birationally dominating $R'_1$.

\vskip 2pt
\noindent
{\bf Step 1.}
 Let $R'_1$ have maximal ideal $\m'_1 = (x_1,  y_1)R_1$,  where $x_1 = x,   ~ y_1 = (y/x) - 1$ .
 Define $v(y_1) = 1/2$.  
 
 \vskip 2pt

 \noindent
 {\bf Step 2.}
 The  local quadratic
 transform $R'_2$ of $R'_1$  has maximal ideal $\m'_2$ generated by 
 $x_2 = x_1/y_1, ~y_2 = y_1$.
 We have $v(x_2) = 1/2$,  $v(y_2) = 1/2$.  
 
 \vskip 2pt

 \noindent
 {\bf Step 3.} 
 Define $y_3 = (y_2/x_2) - 1$ and 
 assign $v(y_3) = 1/4$.  Then $x_3 = x_2$, $v(x_3) = 1/2$.
 
\vskip 2pt

 \noindent
 {\bf Step 4.}   
  The    local quadratic   transform $R'_4$ of $R'_3$ has maximal ideal $\m'_4$ generated by
   $x_4 = x_3/y_3, ~y_4 = y_3$.  Then $v(x_4) = v(y_4) = 1/4$.
   
 Continuing this  2-step process yields an infinite directed union $(R'_n,\m'_n)$ of local 
 quadratic transforms of 2-dimensional RLRs.  Let $V' = \bigcup_{n \ge 0}R'_n$.
 Then $V'$ is a valuation ring by \cite[Lemma~12]{Abh}.  
 Let $v'$ be a valuation associated to $V'$ such that $v'(x)= 1$.  Then $v'(y) = 1$ 
 and $v'$ takes the same rational values on the  generators of $\m'_n$ as 
 defined by $v$.  Since there are infinitely many translations as described in Steps $2n + 1$
 for each  integer $n \ge 0$, it follows that $V'$ has rational rank 1,  e.g.,  
 see \cite[Remark~5.1(4)]{HKT2}.  
 
 The multiplicity values of $\{R'_n,\m'_n)\}$  are   $1, \frac{1}{2}, \frac{1}{2},
   \frac{1}{4},  \frac{1}{4},  \frac{1}{8},  \frac{1}{8} \ldots $,  the sum of which
   converges to 3.
   
   Define $V = V'(\frac{z}{x^2y^2})$,  the localization of the polynomial ring $V'[\frac{z}{x^2y^2}]$ at the 
   prime ideal $\m_{V'}V'[\frac{z}{x^2y^2}]$.  One sometimes refers to $V$ as a Gaussian or trivial or Nagata
   extension of $V'$ to a valuation ring on the simple transcendental field extension generated by 
   $\frac{z}{x^2y^2}$ over $k(x, y)$.  It follows that $V$ has the same value group as $V'$ and the residue
   field of $V$ is a simple transcendental extension  of the residue field of $V'$  that is generated by the 
   image of $\frac{z}{x^2y^2}$ in $V/\m_V$. 
   
    Let $v$ denote the 
 associated valuation to $V$  such that   $v(x) = 1$.  It follows that $v(y) = 1$ and $v(z) = v(x^2y^2) = 4$. 
 Let $R_0 = k[x, y, z]_{(x, y, z)}$.  Then $R_0$ is birationally dominated by $V$. Let
 $\{(R_n, \m_n)\}_{n \ge 0}$ be the sequence of local quadratic transforms of $R_0$ along $V$.
 
 We describe the first few steps:

\vskip 2pt
\noindent
{\bf Step 1.}
  $R_1$ has maximal ideal $\m_1 = (x_1,  y_1, z_1)R_1$,  where $x_1 = x,   ~ y_1 = (y/x) - 1$, and $z_1 = z/x$.
  Also  $v(y_1) = 1/2$.

 \vskip 2pt
 
 \noindent
 {\bf Step 2.}
 The  local quadratic
 transform $R_2$ of $R_1$ along $V$ has maximal ideal $\m_2$ generated by $x_2 = x_1/y_1, ~y_2 = y_1$ and 
 $z_2 = z_1/y_1$.    We have $v(x_2) = 1/2$,  $v(y_2) = 1/2$ and $v(z_2) = 4 - 3/2 > 3/2$.  
 
 \vskip 2pt
 
 \noindent
 {\bf Step 3.}  The local quadratic transform $R_3$  of $R_2$ along $V$ has $y_3 = (y_2/x_2) - 1$, where 
 $v(y_3) = 1/4$, and  $x_3 = x_2$, $v(x_3) = 1/2$ and $v(z_3) > 1/2$.
 
\vskip 2pt
 
 \noindent
 {\bf Step 4.}   
  The    local quadratic
 transform $R_4$ of $R_3$ along $V$ has maximal ideal $\m_4$ generated by $x_4 = x_3/y_3, ~y_4 = y_3$ and 
  $z_4 = z_3/y_3$. 
  
  The multiplicity values of the sequence  $\{(R_n, \m_n)\}_{n \ge 0}$ along $V$ are  the same as that 
   for  $\{R'_n,\m'_n)\}$, namely 
  $1, \frac{1}{2}, \frac{1}{2},  \frac{1}{4},  \frac{1}{4},  \frac{1}{8},  \frac{1}{8} \ldots $.   
   Let $S = \bigcup_{n \ge 0}R_n$. Since $S$ is birationally dominated by the rank 1 valuation ring $V$,
   it follows that $S$ is an archimedean Shannon extension.  Since we never divide in the $z$-direction, we have $S \subseteq R_{zR}$,  and $S$ is 
  not a valuation ring. 
 \end{example}

\section{Quadratic Shannon extensions along a prime ideal}\label{section 3.5}

Let $R$ be a Noetherian local domain 
and let $\{ R_n \}_{n \ge 0}$ be an infinite sequence of LQTs of $R = R_0$.
Using the terminology of Granja and Sanchez-Giralda \cite[Definition~3 and Remark~4]{Gra2},
for a prime ideal $P$ of $R$,
we say the quadratic sequence $\{R_n\}$ is along $R_P$ 
if $\bigcup_{n \ge 0} R_n \subseteq R_P$. 

Let $P$ be a nonzero, nonmaximal prime ideal of a Noetherian local domain $(R, \m)$.
Proposition~\ref{3.5} establishes a one-to-one correspondence 
between sequences $\{R_n\}$ of LQTs of $R = R_0$ along $R_P$ 
and sequences $\{ \overline{R_n} \}$ of LQTs of $\overline {R_0} = R/P$.

\begin{proposition}\label{3.5}
	Let $R$ be a Noetherian local domain and let $P$ be a nonzero nonmaximal prime ideal of $R$.
	Then there is a one-to-one correspondence between:
	
	\begin{description}
		\item[{\em (1)}]  Infinite sequences $\{ R_n \}_{n \ge 0}$ of LQTs  of $R_0 = R$ along $R_P$.
		\item[{\em (2)}] Infinite sequences $\{ \overline{R_n} \}_{n \ge 0}$ of LQTs  of $\overline{R_0} = R / P$.
	\end{description}
	Given such a sequence $\{ R_n \}_{n \ge 0}$, the corresponding sequence is $\{ R_n / (P R_P \cap R_n) \}$.
	Denote $S = \bigcup_{n \ge 0} R_n$ and $\overline{S} = \bigcup_{n \ge 0} \overline{R_n}$,
	and let $\widetilde{S}$ be the pullback of $\overline{S}$ with respect to the quotient map $R_P \rightarrow \kappa(P)$ as in the following diagram:
	\begin{center}
\begin{tikzcd}
   \widetilde{S} = \alpha^{-1}( \overline{S})\arrow[twoheadrightarrow]{r}\arrow[hookrightarrow]{d} 
  &  \overline{S}\arrow[hookrightarrow]{d}
  \\ 
   R_P\arrow[twoheadrightarrow]{r}{\alpha} 
 &  \kappa(P)
\end{tikzcd}
\end{center}
	Then $\widetilde{S} = S + P R_P$ and $\widetilde{S}$ is non-archimedean.
\end{proposition}

\begin{proof}
	The correspondence follows from   \cite[Corollary II.7.15, p.~165]{H}.
	The fact that $\widetilde{S} = S + P R_P$ is a consequence 
	of the fact that $\widetilde{S}$ is a pullback of $\overline{S}$ and $R_P$.  
	That  $\widetilde{S}$ is non-archimedean is a consequence of the observation that
for each $x \in {\ff m}_{\widetilde{S}} \setminus PR_P$, the fact that $PR_P \subseteq \widetilde{S} \subseteq R_P$ implies
	 $PR_P \subseteq x^k{\widetilde{S}}$ for all $k>0$. 
\end{proof}

\begin{lemma}\label{3.6L}
	Assume notation as in Proposition~\ref{3.5}.
	If $\overline{S}$ is a rank $1$ valuation ring and the multiplicity sequence of $(\overline{R}, \overline{S})$ is divergent, then $S = \tilde{S}$.
\end{lemma}

\begin{proof}
	Let $\nu$ be a valuation for $\overline{S}$ and assume that $\nu$ takes values in ${\mathbb{R}}$.  Let $f \in \tilde{S}$. We claim that $ f\in S$. 
	Since $\tilde{S} = S + P R_P$, we may assume $f \in P R_P$.
	Write $f = \frac{g_0}{h_0}$, where $g_0 \in P$ and $h_0 \in R \setminus P$.
	
	Suppose we have an expression of the form $f = \frac{g_n}{h_n}$,
	where $g_n \in P R_P \cap R_n$ and $h_n \in R_n \setminus P R_P$.
	Write $\m_n R_{n+1} = x R_{n+1}$ for some $x \in \m_n$.
	Since $P R_P \cap R_n \subseteq \m_n$, it follows that $g_n = x g_{n+1}$ for $g_{n+1} = \frac{g_n}{x} \in R_{n+1}$.
	Denote the image of $h \in R_n$ in $\overline{R_n}$ by $\overline{h}$. 
	Since $h_n \in R_n \setminus P R_P$, we have that $\overline{h_n} \ne 0$ and $\nu (\overline{h_n})$ is a finite nonnegative real number.
	If $\nu (\overline{h_n}) > 0$, then $h_n \in \m_n$, so $h_n = x h_{n+1}$ for $h_{n+1} = \frac{h_n}{x} \in R_{n+1}$.
	Thus we have written $f = \frac{g_{n+1}}{h_{n+1}}$, where $g_{n+1} \in P R_P \cap R_{n+1}$ and $h_{n+1} \in R_{n+1} \setminus P R_P$,
	such that $\nu (\overline{h_{n+1}}) = \nu (\overline{h_n}) - \nu (\overline{\m_n})$.
	
	Since $\sum_{n \ge 0} \nu (\overline{\m_n}) = \infty$ and $\nu (\overline{h_0})$ is finite, this process must halt
	with $f = \frac{g_n}{h_n}$ as before such that $\nu (\overline{h_n}) = 0$.
	Since $\nu (\overline{h_n}) = 0$, $\overline{h_n}$ is a unit in $\overline{R_n}$, so $h_n$ is a unit in $R_n$, and thus $f \in R_n$.
\end{proof}

\begin{lemma} \label{3.6}  Let $P$ be a nonzero nonmaximal prime ideal of a regular local ring $R$.
Let $\{ R_n \}_{n \ge 0}$ be a sequence of LQTs of $R_0 = R$ along $R_P$ and let $\{\overline{R_n}\}$ be the induced sequence of LQTs of $\overline{R_0} = R/P$ as in Proposition~\ref{3.5}.
Denote $S = \bigcup_{n \ge 0} R_n$ and $\overline{S} = \bigcup_{n \ge 0} \overline{R_n}$.
Then the following are equivalent:
	\begin{description}
		\item[{\em (1)}] $S$ is the pullback of $\overline{S}$ along the 
		surjective  map $R_P \rightarrow  \kappa (P)$.
		\item[{\em (2)}] The Noetherian hull of $S$ is $R_{P}$.
		\item[{\em (3)}] $\overline{S}$ is a rank $1$ valuation ring and the multiplicity sequence of $(\overline{R}, \overline{S})$ is divergent.
	\end{description}
	If these conditions hold, then $\overline{S}$ has rational rank $1$.
\end{lemma}

\begin{proof} (1) $\implies$ (2):  As a pullback, the quadratic Shannon extension $S$ is non-archimedean (see the proof of Proposition~\ref{3.5}). 
Let $x \in S$ be such that $xS$ is $\m_S$-primary (see Theorem~\ref{hull}). By Theorem~\ref{overview}, the ideal 
$Q = \bigcap_{n \geq 0}x^nS$ is a nonzero prime ideal of $S$, every nonmaximal prime ideal of 
$S$ is contained in $Q$ and $T = S_Q$. Assumption (1) implies that $PR_P$ is a nonzero ideal of both $S$ and $R_P$.
Hence $R_P$ is almost integral over $S$. 
We have $S \subseteq S_Q = T \subseteq R_P$, and $S_Q$ is an  RLR and therefore completely integrally 
closed. It follows that $S_Q = R_P$ is the Noetherian hull of $S$.

\noindent
(2) $\implies$ (3):  Since the Noetherian hull $R_P$ of $S$ is local,  Theorem~\ref{overview} implies
that $S$ is non-archimedean and $PR_P \subseteq S$.  By Theorem~\ref{pullbacks}, $\overline S = S/PR_P$ is a rational
rank 1 valuation ring. 
The valuation $\nu$ associated to $\overline S$  is equal to the valuation
$w'$ of Theorem~\ref{nonarch valuation}. By item 2 of  Theorem~\ref{nonarch valuation}
and item~6 of Theorem~\ref{overview}, we have
 $$\sum_{n=0}^{\infty} \nu (\overline{{\ff m}_n}) = \sum_{n=0}^{\infty} w ({\ff m}_n) = \infty.$$
 
\noindent
(3) $\implies$ (1): This is proved in Lemma ~\ref{3.6L}.

\end{proof}

\begin{remark} 
	Let $P$ be a nonzero nonmaximal prime ideal of $R$ and let $S$ be a non-archimedean quadratic Shannon extension of $R$ with Noetherian hull $R_P$.
	The proof of Lemma~\ref{3.6} shows that the multiplicity sequence of $(R/P,S/PR_P)$ is given by $\{w(\m_i)\}$, where $w$ is as in Definition~\ref{w-function}.

With notation as in Lemma~\ref{3.6},   examples where $\overline{S}$ is a rank 1 valuation ring that is not discrete are given in 
 \cite[Examples~7.11 and  7.12]{HKT2}. 
		
\end{remark}

\begin{theorem}[Existence of  Shannon Extensions]  \label{3.8}
	Let $P$ be a nonzero nonmaximal prime ideal of a regular local ring $R$.
	\begin{description}[(2)]
		\item[\em{(1)}] There exists a non-archimedean quadratic 
		 Shannon extension of $R$ with $R_P$ as its Noetherian hull.
		\item[\em{(2)}]    If there exists an archimedean  quadratic Shannon extension of $R$
		contained in $R_P$,  then  $\dim R / P \ge 2$.
	\end{description}
\end{theorem}

\begin{proof}
To prove item~1,  we use a result of Chevalley that every Noetherian local 
domain is birationally dominated by a DVR
\cite{Che}.  Let $V$ be a DVR birationally dominating $R/P$. 
We apply  Lemma~\ref{3.6} with this $R$ and $P$. Let $\{\overline{R_n}\}$ be the sequence of LQTs
of $\overline{R_0}= R/P$ along $V$.   Let $S$ be the union of the corresponding sequence of LQTs 
of $R$ given by Proposition~\ref{3.5}. 
Proposition~\ref{3.7} implies
that  $\overline S = V$  and Lemma~\ref{3.6} implies that 
$S = \widetilde S$ is a non-archimedean Shannon extension with $R_P$ 
as its Noetherian hull. 

For item~2, if $\dim R/P = 1$,  then $\dim R_P = \dim R - 1$ since an RLR is catenary.  
If $S$ is an archimedean Shannon 
extension of $R$,
then $\dim S \le \dim R - 1$  by \cite[Lemma~3.4 and Corollary~3.6]{HLOST}.  Therefore $R_P$
does not contain the 
 Noetherian hull of an archimedean Shannon   extension of $R$ if $\dim R/P = 1$.   
      \qed
\end{proof}
 
 \begin{discussion} \label{4.6disc} 
 Let $P$ be a nonzero nonmaximal prime of a regular local ring $R$ such that  $\dim R/P \ge 2$.  
 We ask:
 
 \noindent
 {\bf Question}:   Does   there exists an
  archimedean  quadratic Shannon extension of $R$ contained in $R_P$? 
  
  The question reduces to the case where $\dim R/P = 2$,  for if $Q$ is a prime ideal of $R$ with
  $\dim R/Q \ge  2$,  then there exists a prime ideal $P$ of $R$ such that $Q \subseteq P$ and 
  $\dim R/P = 2$.    Then  $R_P \subseteq R_Q$.   Hence   a quadratic Shannon extension of $R$ 
  contained in $R_P$ is contained in $R_Q$.  
  
  Assume that $P$ is a nonzero prime ideal of $R$ such that $\dim R/P = 2$.   It is not difficult to see that 
  the 2-dimensional Noetherian local domain $\overline{R_0} = R/P$ is birationally dominated by a 
  rank 1 valuation domain $V$  of rational rank 2.      Consider the 
 infinite sequence of LQTs $\{ \overline{R_n} \}_{n \ge 0}$ of $\overline{R_0} = R / P$  along $V$ and 
 let $\overline S = \bigcup_{n \ge 0}\overline{R_n}$.   Then $\overline S$ is birationally dominated by $V$.  
 Each of the $\overline{R_n}$ is a 2-dimensional Noetherian local domain and $\dim \overline S $ 
 is either $1$ or $2$.

 Let   $\{R_n\}$  be the sequence of LQTs of $R$ given by Proposition~\ref{3.5} that corresponds 
to $\{\overline{R_n}\}$, and let $S = \bigcup_{n}R_n$.  
Then $\dim R_n > 2$ for all $n$.  Hence  $S$ is  
   a quadratic Shannon extension of $R$ and $S \subseteq R_P$.  Let $\p = PR_P \cap S$.
   Then $S/\p = \overline S$. 
   
    If  $\dim \overline S = 1$  then there  are no prime ideals of $S$
   strictly between $\p$  and  $\m_S$.   Since $V$ has rational rank 2,  the multiplicity 
   sequence of $(\overline R, \overline S)$ is convergent.
  Lemma~\ref{3.6} implies that the Noetherian hull of $S$ is not $R_P$.  Hence if  $\dim \overline S = 1$,
  there exists an    archimedean quadratic Shannon extension $S$ of $R$ contained in $R_P$.  
 \end{discussion}


Theorem~\ref{3.8} implies the following:

\begin{corollary}\label{principal cor}   {\em (Lipman \cite[Lemma 1.21.1]{Lip})}  Let $P$ be  a nonmaximal prime ideal of
a regular local ring $R$.   Then there  exists a  quadratic Shannon extension  of $R$ contained in $R_P$. 
\end{corollary}

In Theorem~\ref{pullback thm}
we use Lemma~\ref{3.6} to characterize the overrings of a regular local ring $R$ 
that are Shannon extensions of $R$ with Noetherian hull $R_P$, where $P$ is a nonzero nonmaximal prime ideal of $R$. Note that by Theorem~\ref{overview} such a Shannon extension is necessarily non-archimedean. 

\begin{theorem}[Shannon Extensions with Specified Local Noe\-th\-erian Hull]   \label{pullback thm}
Let $P$ be a nonzero   nonmaximal prime ideal of  a regular local ring $R$.
%
%
The quadratic Shannon extensions of $R$ with Noetherian hull $R_P$ are precisely the rings $S$ such that $S$ is a pullback along the residue map  $\alpha:R_P \rightarrow \kappa(P)$ 
 of  a rational rank $1$ valuation ring birationally dominating $R/P$ whose multiplicity sequence  is divergent.
\begin{center}
\begin{tikzcd}
   S = \alpha^{-1}( {\cal V})\arrow[twoheadrightarrow]{r}\arrow[hookrightarrow]{d} 
  &  {\cal V}\arrow[hookrightarrow]{d}
  \\ 
   R_P\arrow[twoheadrightarrow]{r}{\alpha} 
 &  \kappa(P)
\end{tikzcd}
\end{center}
\end{theorem}

\begin{proof}
	If $S$ is a quadratic Shannon extension with Noetherian hull $R_P$, then by Lemma~\ref{3.6},  $S$ is a pullback along the map  $R_P \rightarrow \kappa(P)$ 
 of a  rational rank $1$ valuation ring birationally dominating $R/P$ whose multiplicity sequence  is divergent.

	Conversely, let $S$ be such a pullback. 
	Let $\{ \overline{R_n} \}_{n \ge 0}$ denote the sequence of LQTs of $\overline{R_0} = R / P$ along $\mathcal{V}$ and let $\{ R_n \}_{n \ge 0}$ denote the induced sequence of LQTs of $R_0 = R$ as in Proposition~\ref{3.5}.
	Then Lemma~\ref{3.6} implies that $S = \bigcup_{n \ge 0} R_n$,
	so $S$ is a quadratic Shannon extension.
\end{proof}



\begin{corollary}
	Let $P$ be a prime ideal of the regular local ring $R$ with $\dim R / P = 1$.
	\begin{description}[(1)]
    \item[\em{(1)}]
      The quadratic Shannon extensions of  $R$ with Noetherian hull $R_P$ 
      are precisely the pullbacks along the residue map $R_P \rightarrow \kappa (P)$ of the finitely many DVR overrings $\mathcal{V}$ of $R/P$.
    \item[\em{(2)}]
      If $R/P$ is a DVR, then $R+PR_P$ is the unique quadratic Shannon extension of $R$ 
      with Noetherian hull $R_P$.
  \end{description}
\end{corollary}

\begin{proof}
	The Krull-Akizuki Theorem \cite[Theorem 11.7]{Mat}  implies  that $R/P$ has finitely many valuation overrings, 
	each of which is a DVR. By  Theorem~\ref{pullback thm}  there is a one-to-one correspondence between 
	these  DVRs and the Shannon extensions of $R$ with Noetherian hull $R_P$. This proves item 1.  
	If  $R/P$ is a DVR, then by item 1, the pullback $R+PR_P$ of $R/P$ along the map $R_P \rightarrow \kappa(P)$ is the unique quadratic Shannon extension of $R$ with Noetherian hull $R_P$.  This verifies item 2. 
  \qed
\end{proof}

\section{Classification of non-archimedean  Shannon extensions}

In  Theorem~\ref{pullback cor} we classify  the non-archimedean 
quadratic Shannon extensions $S$  that occur as overrings of a given regular local ring   $R$. The classification is extrinsic to $S$ in the sense that a prime ideal of  an iterated quadratic transform of $R$ is needed for the description of the overring $S$ as a pullback.   In Theorem~\ref{function field} we give an intrinsic rather than extrinsic  
characterization of certain of  the non-archimedean quadratic  Shannon extensions with principal maximal ideal 
that occur   in an algebraic function  field of characteristic $0$. In this case, we are able 
 to characterize such rings in terms of pullbacks without the explicit requirement of a regular local ``underring'' of $S$.
  This allows us to give an additional source of examples of
   non-archimedean  quadratic Shannon extensions in Example~\ref{pullback example}.

\begin{theorem}[Classification of non-archimedean  Shannon extensions]  
\label{pullback cor}
 Let $R$ be a regular local ring with $\dim R \geq 2$, and let $S$ be an 
	overring of $R$. Then $S$ is a non-archimedean quadratic Shannon extension of $R$ if and only if 
	there is a ring ${\cal V}$, a nonnegative integer $i$ and a prime ideal $P$ of $R_i$ such that 
	\begin{description}[(a)]
	\item[{\em (a)}] ${\cal V}$ is  a rational rank $1$ 
valuation ring  of $\kappa(P)$ that contains the image of $R_i/P$ in $\kappa(P)$  and has divergent   multiplicity sequence over this image, and 
\item[{\em (b)}] $S$ is a pullback of ${\cal V}$ along the residue map $\alpha:(R_i)_P \rightarrow \kappa(P)$.  
\end{description} 
\begin{center}
\begin{tikzcd}
   {S} = \alpha^{-1}( \cal{V})\arrow[twoheadrightarrow]{r}\arrow[hookrightarrow]{d} 
  &  {\cal V}\arrow[hookrightarrow]{d}
  \\ 
   (R_i)_P\arrow[twoheadrightarrow]{r}{\alpha} 
 &  \kappa(P)
\end{tikzcd}
\end{center}

\end{theorem}
\begin{proof} Suppose $S$ is a non-archimedean quadratic Shannon extension, and let $\{R_i\}$ be the sequence of iterated  QDTs such that $S = \bigcup_{i}R_i$. By Theorem~\ref{overview}, the Noetherian hull $T$ of $S$ is a local ring, and by Theorem~\ref{hull}, there is $i>0$ and a prime ideal $P$ of $R_i$ such that $T = (R_i)_P$.    Since $S$ is a non-archimedean quadratic Shannon extension of $R_i$, 
Theorem~\ref{pullback thm}  implies there is a valuation ring ${\cal V}$ such that (a) and (b) hold for $i$, $P$,  $S$ and ${\cal V}$.  

Conversely, suppose there is a ring  ${\cal V}$, a nonnegative integer $i$ and a prime ideal $P$ of $R_i$ 
 that satisfy (a) and (b).  
 By Theorem~\ref{pullback thm}, $S$ is a quadratic Shannon extension of $R_i$ with Noetherian hull $(R_i)_P$.  Thus $S$ is a quadratic Shannon extension of $R$ that is non-archimedean by Theorem~\ref{overview}.   \qed
\end{proof}


In contrast to Theorem~\ref{pullback cor}, the pullback description in Theorem~\ref{function field} is without reference to a specific regular local underring of $S$. Instead, the 
  proof  constructs one using resolution of singularities. 
 Because our use of this technique  is elementary, we  frame our proof in terms of projective models rather than projective schemes. 
  For more background on projective models, see  \cite[Sections 1.6 - 1.8]{Abh2} and \cite[Chapter VI, \S 17]{ZS}. 
  Let $F$ be a field and  let $k$ be a subfield of $F$.  
  Let $t_0 = 1$ and assume 
  that $t_1,\ldots,t_n$ are nonzero elements of $ F$ such that $F = k(t_1, \ldots, t_n)$.   
 For each $i \in \{0,1,\ldots, n\}$, define $D_i = k[t_0/t_i,\ldots,t_n/t_i]$.  
 The {\it projective model} of $F/k$ with respect to $t_0,\ldots,t_n$ 
 is the collection of local rings given by 
 $$X = \{(D_i)_P:i \in \{0,1,\ldots,n\}, \: P \in \Spec(D_i)\}.$$  

  If $k$ has characteristic $0$, then by resolution of singularities (see for example \cite[Theorem 6.38, p.~100]{Cut})  there is a projective model $Y$ of $F/k$ such that every regular local ring in $X$ is in $Y$,  every local ring in $Y$ is a regular local ring,   and every local ring in $X$ is dominated by  a (necessarily regular) local ring in $Y$.   
  
 By a {\it valuation ring of $F/k$} we mean a valuation ring $V$ with quotient field $F$ such that $k$ is a subring of $V$.

\begin{theorem} \label{function field}
	Let $S$ be a local domain containing as a  subring a field $k$ of characteristic $0$.   Assume that $\dim S \geq 2$ and that the quotient field $F$ of $S$ is a finitely generated extension of $k$. Then  the following are equivalent:
	\begin{description}[(2)]
  	\item[{\em (1)}]
    	$S$ has a principal maximal ideal and $S$ is a quadratic Shannon extension 
    	of a regular local ring $R$  that is essentially finitely generated over $k$.  
  	\item[{\em (2)}] 
  	  There is a regular local overring $A$ of $S$  and a DVR ${\cal V}$ of $(A/\m_A)/k$   such that
	  \begin{description}[(a)]
	  \item[{\em (a)}] 
	   {\rm tr.deg}$_k \: A/{\ff m}_A  + \dim A  = $ {\rm tr.deg}$_k \:F$, 
	and   
	\item[{\em (b)}] 
	  $S$ is the pullback of ${\cal V}$ along the residue map $\alpha:A \rightarrow A/{\ff m}_A$.   
		\end{description}
\end{description} 
\begin{center}
\begin{tikzcd}
   {S} = \alpha^{-1}( {\cal V})\arrow[twoheadrightarrow]{r}\arrow[hookrightarrow]{d} 
  &  {\cal V}\arrow[hookrightarrow]{d}
  \\ 
   A\arrow[twoheadrightarrow]{r}{\alpha} 
 &  A/{\ff m}_A
\end{tikzcd}
\end{center}
\end{theorem}

\begin{proof} 
(1) $\Longrightarrow$ (2):  Let $x \in S$ be  such that ${\ff m}_S = xS$.  
By Theorem~\ref{hull},  $S[1/x]$ is the Noetherian hull of $S$ and $S[1/x]$ is  
 a regular ring.  
Since $\dim S > 1$, the ideal $P = \bigcap_{k>0}x^kS$ is a nonzero  prime ideal of 
$S$ \cite[Exercise~1.5, p.~7]{Kap}.  Hence $S$ is non-archimedean.
By  Theorem~\ref{overview},  $S_P$ is the Noetherian hull of $S$ and hence  $S_P = S[1/x]$. 
Let $A = S_P$ and ${\cal V} = S/P$.  
By Theorem~\ref{pullbacks}, $S$ is a pullback of the DVR ${\cal V}$ with respect to the map
$A \to A/{\ff m}_A$.  
By assumption, $S$ is a quadratic Shannon extension of a regular local ring $R$ that is 
essentially finitely generated over $k$.  For sufficiently large $i$, we have     $A = S_P = (R_i)_{P \cap R_i}$ by \cite[Proposition 3.3]{HLOST}. Since $R_i$ is essentially finitely generated over $R$, and $R$ is essentially finitely generated over $k$, we have that $A$ is essentially finitely generated over $k$.  
  By the Dimension Formula \cite[Theorem~15.6, p.~118]{Mat}, 
  \begin{center} 
{\rm tr.deg}$_k \: A/{\ff m}_A  + \dim A  = $ {\rm tr.deg}$_k \:F$. \end{center}  This completes the proof 
that statement 1 implies statement 2.

(2) $\Longrightarrow$ (1):  Let $P ={\ff m}_A$.   By item 2b,  $P$ is a prime ideal of  $S$, $A = S_P$, $P = PS_P$ and ${\cal V} = S/P$.   Let $x \in {\ff m}_S$ be  such that the image of $x$ in the DVR $S/P$ generates the maximal ideal. Since $P = PS_P$, we have $P \subseteq xS$.  Consequently, ${\ff m}_S = xS$, and so $S$ has a principal maximal ideal.   

To prove that $S$ is a  quadratic Shannon extension of a regular local ring that is essentially finitely generated over  $k$, it suffices by Theorem~\ref{pullback thm} to prove:
\begin{description}[(iii)]
\item[(i)] There is a subring $R$ of $S$ that is  a regular local ring essentially finitely generated over $k$.

\item[(ii)]  $A$ is a localization of $R$ at the  prime ideal $P \cap R$. 
  
\item[(iii)]    ${\cal V}$ is a  valuation overring   of $(R+P)/P$ with   divergent multiplicity sequence. 
\end{description}
 
Since $F$ is a finitely generated field extension of $k$ and $A$ (as a localization of $S$) has quotient field $F$, there is a finitely generated $k$-subalgebra $D$ of $A$ such that  the quotient field of $D$ is $F$.  By item 2a,  
$A/P$ has finite transcendence degree over $k$.  Let $a_1,\ldots,a_n$ be elements of $A$ whose images in $A/P$ form a transcendence basis for $A/P$ over $k$. Replacing $D$ with $D[a_1,\ldots,a_n]$, and defining $p = P \cap D$, we may assume that   $A/P$ is algebraic over $\kappa(p)=D_p/pD_p$. In fact, since the normalization of an affine $k$-domain is again an affine $k$-domain, we may assume also that $D$ is an integrally closed finitely generated $k$-subalgebra of $A$ with quotient field $F$. 
 Since $D$ is a finitely generated $k$-algebra, $D$ is universally catenary. By the Dimension Formula \cite[Theorem 15.6, p.~118]{Mat}, we have \begin{center}$\dim D_{p} + $ tr.deg$_k \: \kappa(p) = $ tr.deg$_k \: F.$\end{center}  Therefore, item 2a implies 
 \begin{center}$\dim D_{p} + $ tr.deg$_k \: \kappa(p)  = \dim A + $ tr.deg$_k \: A/P$.    
 \end{center}
 Since $A/P$ is algebraic over $\kappa(p)$, we conclude that $\dim D_p = \dim A$. 

 The normal ring $A$ birationally dominates  the excellent normal ring $D_p$,  so  $A$ is essentially finitely generated over $D_p$ \cite[Theorem 1]{HHS}. Therefore $A$ is 
 essentially finitely generated over $k$.

  Since $A$ is essentially finitely generated over $k$, 
  the local ring  $A$ is in a projective model $X$ of $F/k$. 
  As discussed before the theorem, resolution of singularities implies that
  there exists  a projective model $Y$ of $F/k$ such that every regular local ring in $X$ is in $Y$,  
  every local ring in $Y$ is a regular local ring,   
  and every local ring in $X$ is dominated by  a local ring in $Y$. 
  

Since $A$ is  a regular local ring in $X$,   $A$ is a local ring in the projective model $Y$.   
Let  $x_0,\ldots,x_n \in F$ be nonzero elements such that 
with $D_i := k[x_0/x_i,\ldots,x_n/x_i]$  for each $i \in \{0,1,\ldots,n\}$, we have 
 $$Y = \bigcup_{i =0}^n \:\{(D_i)_Q: Q \in \Spec(D_i)\}.$$  
 Since $S$ has quotient field $F$,  
we may assume that $x_0,\ldots,x_n \in S$. 
 Since 
$A $ is in $ Y$,  there is  $i \in \{0,1,\ldots,n\}$ such that $A = (D_i)_{P \cap D_i}$.

By item 2b,  ${\cal V}=S/P$ is a valuation ring with quotient field $A/P$. For $a \in A$, let $\overline a$ denote
the image of $a$ in the field $A/P$.  Since $S/P$ is a valuation ring of $A/P$, there exists  $j \in \{0,1,\ldots,n\}$ 
 such that 
\begin{equation} \label{factor equation} 
(~\{\overline{x_k/x_i}\}_{k=0}^n ~)(S/P) ~ = ~(\overline{x_j/x_i})(S/P).
\end{equation}
Notice that $x_i/x_i = 1 \notin P$. Hence at least one of the $x_k/x_i \notin P$,  and
Equation~\ref{factor equation} implies $x_j/x_i \not \in P$.  
Since  
$A = S_P$ and $P = PS_P$,   
  every fractional ideal of $S$ contained in $A$ is comparable to $P$ with respect to set inclusion.  
  Therefore $P \subsetneq (x_j/x_i)S.$ 
  This and Equation~\ref{factor equation} imply that 
 \begin{equation} \label{eq2}
  (x_0/x_i,\ldots,x_n/x_i)S = (x_j/x_i)S.
  \end{equation}
   Multiplying both sides of Equation~\ref{eq2} by $x_i/x_j$ we obtain  
 $$ 
 D_j ~ =  ~k[x_0/x_j,\ldots,x_n/x_j]~ \subseteq ~ S.$$  Let $R = (D_j)_{{\ff m}_S \cap  D_j}.$
Since $Y$ is a nonsingular model, $R$ is a regular local ring with $R \subseteq S \subseteq A$. 

We observe next that $A = R_{P \cap R}$.  
Since $R \subseteq A$, we have that $A$ dominates the local ring $A':=R_{P \cap R}.$ The local ring $A'$ is a member of the projective model $Y$, and  every valuation ring dominating the local ring $A$ in $Y$ dominates also the local ring $A'$ in $Y$.     
Since $Y$ is a projective model of $F/k$,  
the Valuative Criterion for Properness  \cite[Theorem II.4.7, p.~101]{H} implies no two distinct  local rings in $Y$ are dominated by the same valuation ring. Therefore, $A = A'$, so that $A =  R_{P \cap R}$.    

Finally, 
 observe that since ${\cal V} = S/P$ is a DVR overring of $(R+P)/P$,  the multiplicity sequence of  $S/P$ over $(R+P)/P$ is divergent. 
By Theorem~\ref{pullback thm}, $S$ is a quadratic Shannon extension of $R$ with 
Noetherian hull $A=R_{P \cap R}$.  By Theorem~\ref{overview}, $S$ is  non-archimeean, 
 so the proof is complete. \qed
\end{proof}

As an application of Theorem~\ref{function field}, we describe for a finitely generated field extension $F/k$ of characteristic $0$  the valuation rings  with principal maximal ideal that arise as quadratic Shannon extensions of regular local rings that are essentially finitely generated over $k$, i.e., the valuation rings on the Zariski-Riemann surface of $F/k$  that arise from desingularization followed by infinitely many successive quadratic transforms of projective models.  Recall that a valuation ring $V$ of $F/k$ is a {\it divisorial} valuation ring if \begin{center}tr.deg$_k \: V/\m_V = $ tr.deg$_k \: F -1 $.\end{center} Such a valuation ring is necessarily a DVR (apply, e.g., \cite[Theorem 1]{Abh}).  

\begin{corollary} Let $F/k$ be a finitely generated field extension where $k$ has characteristic $0$, and let $S$ be a valuation ring of $F/k$ with principal maximal. 
\begin{description}[$(1)$]
\item[{\em (1)}] Suppose rank $S = 1$. Then there is a sequence $\{R_i\}$ (possibly finite)  of LQTs of a regular local ring $R$ essentially finitely type over $k$ such that $S = \bigcup_i R_i$.  This sequence is finite if and only if $S$ is a divisorial valuation ring. 

\item[{\em (2)}] Suppose rank $S >1$. Then 
 $S$ is a quadratic Shannon extension of a regular local ring essentially finitely generated over $k$ if and only if $S$ has rank $2$ and is contained in a divisorial valuation ring of $F/k$. 
 \end{description}
\end{corollary}

\begin{proof} For item 1, assume rank $S = 1$. By resolution of singularities, there is a nonsingular projective model $X$ of $F/k$ with function field $F$. Let $R$ be the regular local ring in $X$ that is dominated by  $S$. Let $\{R_i\}$ be the sequence of LQTs of $R$ along $S$. If $\{R_i\}$ is finite, then $\dim R_i = 1$ for some $i$, so that $R_i$ is a DVR. Since $S$ is a DVR between $R_i$ and its quotient field, we have $R_i = S$. Otherwise, if $\{R_i\}$ is infinite, then Proposition~\ref{3.7} implies $S = \bigcup_i R_i$ since $S$ is a DVR.    That the sequence is finite if and only if $S$ is a divisorial valuation ring follows from \cite[Proposition 4]{Abh}. 

For item 2, suppose rank $S>1$. Assume first that $S$ is a Shannon extension of a regular local ring essentially finitely generated over $k$. By   \cite[Theorem~8.1]{HLOST}, $\dim S = 2$.  By Theorem~\ref{function field}, $S$ is a contained in a regular local ring $A \subseteq F$ such that $A/\m_A$ is the quotient field of a proper homomorphic image of $S$ and  \begin{equation}\label{eq3} {\rm tr.deg}_k \: A/\m_A + \dim A = {\rm trdeg}_k \: F.\end{equation}

 We claim $A$ is a divisorial valuation ring of $F/k$.  Since
 $A/\m_A$ is the quotient field of a proper homomorphic image of $S$, it follows  that \begin{equation}\label{eq4} {\rm tr.deg}_k \: A/\m_A < \:
{\rm trdeg}_k F.\end{equation} From equations~\ref{eq3} and~\ref{eq4} we conclude that $\dim A \geq 1$.  
  As an overring of the valuation ring $S$, $A$ is also a valuation ring.
Since $A$ is a regular local ring that is not a field, it follows that $A$ is a DVR.  Thus  $\dim A = 1$ and equation~\ref{eq3} implies that \begin{center} tr.deg$_k \: A/\m_A = $ trdeg$_k \: F - 1,$\end{center} which proves that $A$ is a divisorial valuation ring.  

Conversely, suppose rank $S = 2$ and $S$ is a quadratic Shannon extension of a regular local ring that is essentially finitely generated over $k$. Theorem~\ref{function field} and rank~$S = 2$ imply $S$ is contained in a regular local ring $A$ with $\dim A  = 1$  and 
\begin{center} tr.deg$_k \: A/\m_A + 1 = $ trdeg$_k \: F$. \end{center} Thus $A$ is a divisorial valuation ring.

 Finally, suppose rank $S = 2$ and 
$S$ is contained in a divisorial valuation ring $A$ of $F/k$. Since $S$ is a valuation ring of rank 2 with principal maximal ideal it follows that ${\ff m}_A \subseteq S$ and $S/{\ff m}_A$ is DVR. 
Since $A$ is a divisorial valuation ring, we have  \begin{center}tr.deg$_k \: A/\m_A + \dim A = $ trdeg$_k \: F$.\end{center}
 As  a DVR, $A$ is a regular local ring, so  Theorem~\ref{function field} implies $S$ is a quadratic Shannon extension of a regular local ring that is essentially finitely generated over $k$. 
 \qed
\end{proof}

\begin{example} \label{pullback example}  Let  $k$ be a field of characteristic $0$,  let $x_1,\ldots,x_n,y_1,\ldots,y_m$ be algebraically independent over $k$, and let $$A = k(x_1,\ldots,x_n)[y_1,\ldots,y_m]_{(y_1,\ldots,y_m)}.$$  Let $\alpha:A \rightarrow k(x_1,\ldots,x_n)$ be the canonical residue map. 
Then for every DVR $V$ of $k(x_1,\ldots,x_n)/k$, the ring $S = \alpha^{-1}(V)$ is 
by Theorem~\ref{function field} a quadratic Shannon extension of a regular local ring that is essentially finitely generated over $k$. As in the proof that statement 2 implies statement 1 of Theorem~\ref{function field}, the Noetherian hull of $S$ is $A$.  

Conversely, suppose $S$ is 
a
$k$-subalgebra of $F$ with principal maximal ideal  such that $S$ is a quadratic Shannon extension of a regular local ring that is essentially finitely generated over $k$ and $S$ has Noetherian hull $A$.  As in  the proof that statement 1 implies statement 2 of  
 Theorem~\ref{function field}, there is a DVR $V$ of $k(x_1,\ldots,x_n)/k$ such that $S = \alpha^{-1}(V)$.  

It follows that there is a one-to-one correspondence between the DVRs of $k(x_1,\ldots,x_n)/k$ and the  quadratic Shannon extensions $S$ of regular local rings that are essentially finitely generated over $k$, have Noetherian hull $A$, and have a principal maximal ideal.
\end{example}

Theorem~\ref{function field} concerns quadratic Shannon extensions of regular local rings that are essentially finitely generated over $k$. 
Example~\ref{5.3ex} is  a quadratic Shannon extension of a regular local ring $R$ in a function field for which $R$ is not essentially finitely generated over $k$.

\begin{example}  \label{5.3ex}  Let $F = k(x,y, z)$, where $k$ is a field and $x, y, z$ are algebraically  independent over $k$. 
 Let $\tau \in  xk[[x]]$ 
be a formal power series in $x$ such that $x$ and $\tau$ are algebraically independent over $k$.  
Set  $y = \tau$ and define
$V = k[[x]] \cap k(x, y)$.   Then $V$ is a DVR on the field $k(x, y)$ with maximal ideal $xV$ and residue field $V/xV = k$.   
Let $V(z) =  V[z]_{xV[z]}$.  Then $V(z)$ is a DVR on the field $F$ with residue field $k(z)$,  and  
$V(z)$ is not essentially finitely generated over $k$. Let  $R = V[z]_{(x, z)V[z]}$.
Notice that $R$ is a  2-dim RLR. 
  The pullback diagram  of type $\square^*$ 
\begin{center}
\begin{tikzcd}
   {S} = \alpha^{-1}(k[z]_{zk[z]})\arrow[twoheadrightarrow]{r}\arrow[hookrightarrow]{d} 
  &  k[z]_{zk[z]}\arrow[hookrightarrow]{d}
  \\ 
   V(z)\arrow[twoheadrightarrow]{r}{\alpha} 
 &  k(z)
\end{tikzcd}
\end{center}
defines a rank~2 valuation domain $S$ on $F$ that is by Theorem~\ref{pullback cor} a quadratic Shannon extension of  $R$.  For each positive integer $n$,  define 
$R_n = R[\frac{x}{z^n}]_{(z, \frac{x}{z^n})R[\frac{x}{z^n}]}$. 
Then $S = \bigcup_{n \ge 1}R_n$. 
\end{example}

\section{Quadratic Shannon extensions and GCD domains} \label{section 6}

As an application of the pullback description of non-archimedean quadratic Shannon extensions given in Section 4, 
we show in Theorem~\ref{quadratic GCD}  that 
 a quadratic Shannon extension $S$  is coherent, a GCD domain or a finite conductor domain if and only if $S$ is a valuation domain. We extend this fact to all  quadratic Shannon extensions $S$, regardless of 
 whether $S$ is archimedean, by applying structural results for archimedean
 quadratic  Shannon extensions from \cite{HLOST}.   

\begin{definition}  \label{3.65}
 Following McAdam in \cite{McAdam},  an
integral domain $D$ is a {\it finite conductor domain} if 
for elements $a, b$ in the field of fractions of $D$, the 
$D$-module $aD \cap bD$ is finitely generated.  A ring is said to be {\it coherent} if every 
finitely generated ideal is finitely presented.  Chase \cite[Theorem~2.2]{Chase} proves that an 
integral domain $D$ is coherent if and only if the intersection of two finitely generated ideals of $D$ 
is finitely generated. Thus a coherent domain is a finite conductor domain.
 An integral domain $D$ is a   {\it GCD domain} if  for all $a, b \in D$, $aD \cap bD$ is a principal 
ideal of $D$ \cite[page~76 and Theorem~16.2, p.174]{Gil}.  It is clear from the definitions that a 
GCD domain is a finite conductor domain.
\end{definition}

 Examples of  GCD domains  and  finite conductor domains that are not coherent are 
given by Glaz in \cite[Example~4.4  and Example~5.2]{Glaz2} and by 
Olberding and Saydam in \cite[Prop. 3.7]{OS}.   Every Noetherian integral domain 
is coherent,  and a Noetherian domain $D$ is a GCD domain if and only if it is a UFD. 
Noetherian domains that are not UFDs are examples of coherent domains that are 
not GCD domains.

\begin{theorem} \label{quadratic GCD}   Let $S$ be a quadratic Shannon 
extension of a regular local ring.  The following are equivalent:
\begin{description}[(2)]
\item[{\em (1)}] $S$ is coherent.
\item[{\em (2)}] $S$ is a GCD domain. 
\item[{\em (3)}] $S$ is a finite conductor domain.
\item[{\em (4)}] $S$ is a valuation domain.
\end{description}
\end{theorem} 

\begin{proof}
It is true in general that if $S$ is a valuation domain, then $S$ satisfies each of the first 3 items.
As noted above, if $S$ is coherent or a GCD domain, then $S$ is a finite
conductor domain.
To complete the proof of Theorem~\ref{quadratic GCD}, 
it suffices to show that if $S$ is not a valuation domain, 
then $S$ is not a finite conductor domain. Specifically, we assume $S$ is not a valuation domain and we consider three cases. In each case, we find a pair of principal fractional ideals of $S$ that is not finitely generated.

{ \bf Case 1: }
$S$ is non-archimedean.
By Theorem~\ref{overview}, there is a unique dimension $1$ prime ideal $Q$ of $S$, $Q S_Q = Q$ and $S_Q$ is the Noetherian hull of $S$.  If $\dim S_Q = 1$, then, as a regular local ring, $S_Q$ is a DVR; this, along with the fact that $Q = QS_Q$ implies  $S$ is a DVR, contrary to the assumption that $S$ is not a valuation domain. Therefore $\dim S_Q \geq 2$, and 
 there exist elements $f, g \in Q$ that have no common factors in the UFD $S_Q$.
Consider $I = f S \cap g S$,
let $x \in \m_S$ such that $\sqrt{x S} = \m_S$ (see Theorem~\ref{hull}),
and let $a \in I$.
Since $a \in f S_{Q} \cap g S_{Q} = f g S_{Q}$, 
we can write $a = f g y$ for some $y \in S_{Q}$. Now $g y \in Q S_{Q} = Q$ and $f y \in Q S_{Q} = Q$, so $\frac{g y}{x} \in S$ and $\frac{f y}{x} \in S$. Thus 
 $\frac{a}{x} = f \frac{g y}{x} = g \frac{f y}{x} \in I$. This shows 
that $x I = I$ and so $\m_S I = I$.
Since $I \ne (0)$, Nakayama's Lemma implies that $I$ is not finitely generated. 

{ \bf Case 2: } 
$S$ is archimedean, but not completely integrally closed.
By Theorem~\ref{hull}, $\dim S \ge 2$.
We claim that $\m_S $ is not  finitely generated    as an ideal of $S$.   Since $\dim S > 1$,   if $\m_S$ is 
a  principal  ideal, then  $\bigcap_{i}{\ff m}_S^i$ is a nonzero prime ideal of $S$, a contradiction 
to the assumption that $S$ is archimedean. Thus $\m_S$ is not principal.
By \cite[Proposition 3.5]{HLOST}, this implies $\m_S^2 = \m_S$.
From Nakayama's Lemma it follows that $\m_S$ is not finitely generated. 
 Since $S$ is not completely integrally closed, there is an almost integral element $\theta$ over $S$ 
that is not in $S$.  
By \cite[Corollary 6.6]{HLOST}, $\m_S = \theta^{-1} S \cap S$.

{ \bf Case 3: } 
$S$ is archimedean and completely integrally closed.
By Theorem~\ref{hull}, $\dim S \geq 2$. 
By Theorems~\ref{hull} and \ref{complete intcl}, 
$S = T \cap W$, 
where $W$ is the rank~1 nondiscrete valuation ring 
with associated valuation $w (-)$ as in Definition~\ref{w-function}
and $T$ is a UFD that is a localization of $S$.
Since $\sum_{n \ge 0} w (\m_n) < \infty$ by Theorem~\ref{overview},
and since $\m_n S$ is principal and generated by a unit of $T$ for $n \gg 0$,
the $w$-values of units of $T$ generate a non-discrete subgroup of $\mathbb{R}$.

Since $S$ is archimedean, Theorem~\ref{overview} implies $T$ is a non-local UFD. 
Therefore there exist elements $f, g \in S$ that have no common factors in $T$.
As in Case 1, we consider $I = f S \cap g S$.
Since $S = T \cap W$, it follows that
\begin{align*}
  I
    &= (f T \cap g T) \cap (f W \cap g W) \\
    &= f T \cap g T \cap \{ a \in W \mid w (a) \ge \max \{ w (f), w (g) \} \}.
\end{align*}
Assume without loss of generality that $w (f) \ge w (g)$.

For $a \in I$, write $a = (\frac{a}{f}) f$ in $S$ and consider $w (a)$.
Since $\frac{a}{f}$ is divisible by $g$ in $T$, it is a non-unit in $T$, and thus it is a non-unit in $S$.
Since $W$ dominates $S$, it follows that $w (\frac{a}{f}) > 0$ and thus $w (a) > w (f)$.

We claim that $\m_S I = I$.
Since the $w$-values of the units of $T$ generate a non-discrete subgroup of $\mathbb{R}$, for any $\epsilon > 0$, there exists an unit $x$ in $T$ with $0 < w (x) < \epsilon$.
Then for $a \in I$ and for some $x$ with $0 < w (x) < w (a) - w (f)$, we have $\frac{a}{x} \in I$ and thus $a \in \m_S I$.
Since $\m_S I = I$ and $I \ne (0)$, Nakayama's Lemma implies that $I$ is not finitely generated.

In every case, we have constructed a pair of principal fractional ideals of $S$ whose intersection is not finitely generated.
We conclude that if $S$ is not a valuation domain, then $S$ is not a finite conductor domain. \qed


\end{proof}

\end{document}